\documentclass[reqno,final,a4paper]{amsart}
\usepackage{color}
\usepackage{amsmath, amssymb, amsthm, thmtools}
\usepackage[colorlinks=true,allcolors=blue
]{hyperref}
\usepackage{mathrsfs}
\usepackage{mathtools}
\usepackage[linesnumbered,ruled,vlined]{algorithm2e}

\SetCommentSty{mycommfont}
\SetKwInOut{Input}{input}\SetKwInOut{Output}{output}

\usepackage[noadjust]{cite}
\usepackage[noabbrev,capitalize,nameinlink]{cleveref}
\crefname{equation}{}{}
\usepackage{fullpage}
\usepackage{graphics}
\usepackage{tikzscale}
\usepackage{pifont}
\usepackage{tikz}
\usetikzlibrary{fit, backgrounds, shapes, calc}
\usepackage{pgfplots}
\pgfplotsset{compat=1.12}
\usepackage{bbm}
\usepackage[T1]{fontenc}
\usetikzlibrary{arrows.meta}

\usepackage{environ}
\usepackage{framed}
\usepackage{url}

\usepackage[noend]{algpseudocode}
\usepackage[labelfont=bf]{caption}
\usepackage{framed}
\usepackage[framemethod=tikz]{mdframed}
\usepackage{appendix}
\usepackage{graphicx}
\usepackage[textsize=tiny]{todonotes}
\usepackage{tcolorbox}
\usepackage{xcolor}
\usepackage[shortlabels]{enumitem}
\crefformat{enumi}{#2#1#3}
\crefrangeformat{enumi}{#3#1#4 to~#5#2#6}
\crefmultiformat{enumi}{#2#1#3}
{ and~#2#1#3}{, #2#1#3}{ and~#2#1#3}
\allowdisplaybreaks

\usepackage{ifthen}
\numberwithin{equation}{section}
\newtheorem{theorem}{Theorem}[section]
\newtheorem{proposition}[theorem]{Proposition}
\newtheorem{lemma}[theorem]{Lemma}

\crefname{claim}{Claim}{Claims}

\newtheorem{corollary}[theorem]{Corollary}

\newtheorem*{question*}{Question}

\theoremstyle{definition}
\newtheorem{definition}[theorem]{Definition}

\newtheorem*{definition*}{Definition}

\newtheorem*{fact*}{Fact}

\crefname{fact}{Fact}{Facts}

\theoremstyle{remark}

\renewcommand{\Pr}{\mathbb{P}}
\newcommand{\E}{\mathbb{E}}
\newcommand{\eps}{\varepsilon}

\newcommand{\Bin}{\operatorname{Bin}}

\newcommand{\dtv}{d_{\mathrm{TV}}}

\def\epsilon{\varepsilon}

\let\originalleft\left
\let\originalright\right
\renewcommand{\left}{\mathopen{}\mathclose\bgroup\originalleft}
\renewcommand{\right}{\aftergroup\egroup\originalright}

\title{Note on the trace of random walks on pseudorandom graphs}

\author{Yaobin Chen}\address{Shanghai Center for Mathematical Sciences,~Fudan University,~Shanghai,~200438,~China.}\email{ybchen21@m.fudan.edu.cn}

\author{Yiting Wang}\address{Institute of Science and Technology Austria,~Klosterneuburg,~3400,~Austria.}\email{yiting.wang@ist.ac.at}

\thanks{Yiting Wang is supported by the European Research Council (ERC), via grant agreements ``RANDSTRUCT'' No.\ 101076777.}

\date{\today}

\begin{document}

\begin{abstract}
We study the graph-theoretic properties of the trace of random walks on pseudorandom graphs. We show that for any $\varepsilon>0$, there exists a constant $C$ such that the cover time of an $(n,d,\lambda)$-graph $G$ with $d/\lambda\ge C$ is at most $(1+\varepsilon)n\log n$, meaning the expected number of steps needed to reach all vertices at least once is  at most $(1+\eps)n\log n$ regardless of the starting vertex. Furthermore, we prove that with high probability, the trace of a random walk of length $(1+\varepsilon)n\log n$ on $G$ is Hamiltonian, regardless of the starting vertex. These results also hold for random $d$-regular graphs with sufficiently large $d$. These findings answer two questions proposed by Frieze, Krivelevich, Michaeli, and Peled [PLMS, 2018]. Notably, our results imply a bound on a stronger version of the cover time: with high probability, all vertices are covered after $(1+\varepsilon)n\log n$ steps, regardless of the starting vertex. Our proofs rely on the spectral properties of the adjacency matrix and the graph expansion. All results are asymptotically optimal.
\end{abstract}
\maketitle

\section{Introduction}

Let $G$ be a simple undirected connected graph. A \emph{simple random walk} on $G$ is a stochastic process where, starting from a given vertex, each step consists of moving to a neighbor chosen uniformly at random. The set of edges traversed by this walk is called the \emph{trace}. Random walks are fundamental objects in combinatorics, probability theory, and theoretical computer science. For a comprehensive treatment, we refer the reader to the textbook by Levin and Peres~\cite{Levin-Peres-book} or the survey by Lov\'{a}sz~\cite{Lovasz-survey}.

In this paper, we study the graph formed by the trace of a simple random walk on a finite graph. This line of research was initiated by Frieze, Krivelevich, Michaeli, and Peled in~\cite{Frieze-et-al}, who focused on the case where the base graph is random. Motivated by these results, in~\cite{Frieze-et-al}, they asked to extend the inquiry to deterministic graphs that exhibit random-like properties, known as pseudorandom graphs. A prominent class of pseudorandom graphs is that of spectral expanders, or $(n, d, \lambda)$-graphs. An $(n, d, \lambda)$-graph is defined as a $d$-regular graph on $n$ vertices where the second largest eigenvalue in absolute value is at most $\lambda$. The expander mixing lemma shows that $\lambda$ governs the edge distribution: a smaller $\lambda$ implies that the edge distribution of $G$ closely resembles that of the random graph $G(n, d/n)$. For a detailed introduction, we refer the reader to the survey by Krivelevich and Sudakov~\cite{Krivelevich-Sudakov-survey}.

We focus our attention on the case where $G$ is an $(n,d,\lambda)$-graph. Let $\Gamma$ denote the subgraph on the vertex set of $G$ induced by the trace of the random walk. A fundamental question concerns the structural properties of $\Gamma$, particularly its connectivity. This inquiry is intrinsically linked to the \emph{cover time}, $T(G)$, defined as the expected number of steps required to visit every vertex, maximized over all possible starting vertices. Equivalently, $T(G)$ corresponds to the expectation of the hitting time for the property that $\Gamma$ becomes connected. The cover time is a classical parameter in the study of random walks. A seminal result by Feige~\cite{Feige-lower} establishes a universal lower bound of $(1-o(1))n\log n$ for any connected graph on $n$ vertices. For Erd\H{o}s-R\'{e}nyi graph, random regular graph and more generally random graph with a fixed degree sequence, asymptotically sharp bounds have been proven~\cite{Adullah-Copper-Frieze-degree-sequence,Cooper-Frieze-sparse-random, Cooper-Frieze-random-regular,Jonasson, Cooper-Frieze-Lubetzky}. In particular, Frieze, Krivelevich, Michaeli, and Peled~\cite{Frieze-et-al} demonstrated that for random graphs $G(n,p)$ and $\eps>0$, there exists a constant $C_{\eps}$ such that if $p \ge C_{\eps}\log n/n$, the trace $\Gamma$ generated by a random walk of length $(1+\varepsilon)n\log n$ is not merely connected, but Hamiltonian with high probability.

Given that $(n,d,\lambda)$-graphs exhibit quasi-random behavior analogous to the random graph $G(n, d/n)$, it is natural to inquire whether the structural properties of the random walk trace keep in this setting. Specifically, we ask whether a random walk of length $(1+\varepsilon)n\log n$ yields a Hamiltonian trace on an $(n,d,\lambda)$-graph. This question was raised by Frieze, Krivelevich, Michaeli, and Peled~\cite{Frieze-et-al}.  
Recall that the connectivity threshold for $G(n,p)$ is $(1+ o(1))\log n/n$ where as for the trace to be Hamiltonian, one needs $C_\eps\log /n$ where $C_\eps$ is a big constant depending on $\eps$. For $(n,d,\lambda)$-graph, the connectivity holds when $d/\lambda > 1$. Thus, analogously, one would expect the answer holds when $d/\lambda > C_\eps$ for some big constant $C_\eps$.

\subsection{Our results}
By the expander mixing lemma, $(n,d,\lambda)$-graphs represent a class of graphs with expansion properties. The asymptotic behavior of random walks in this setting was investigated in the foundational work of Broder and Karlin~\cite{Broder-Kralin}, who established that for any $d$-regular expander $G$, the cover time satisfies $T(G) = \Theta(n\log n)$ \footnote{The notion of expander in~\cite{Broder-Kralin} is different, but it applies to the $(n,d,\lambda)$ setting.}. Our first contribution is a refinement of this bound for $(n,d,\lambda)$-graphs. We demonstrate that, provided the spectral ratio is sufficiently large, the cover time is asymptotically optimal: $T(G) = (1+o(1))n\log n$.
\begin{proposition}\label{prop:cover-time}
    For every $\varepsilon >0$, there exists $C = C(\eps)$ such that if $G$ be an $(n,d,\lambda)$-graph with $d/\lambda \geq C$, then $T(G) \leq (1+\eps) n\log n$.
\end{proposition}

We now introduce a strengthrening of the cover time, termed the \emph{strong cover time}. Let $ST(G)$ denote the minimum integer $t$ such that a simple random walk of length $t$ visits every vertex of $G$ with high probability (abbrviated as \textbf{whp} throughout the paper), uniformly over all starting vertices. Note that by definition, $ST(G)$ provides a high-probability guarantee, whereas $T(G)$ captures only the expectation. 
Deterministically $ST(G)\geq T(G)$.
In the context of random graphs $G(n,p)$, Jonasson~\cite{Jonasson} established that $ST(G)= (1+o(1))n\log n$ provided that $p = \omega(\log^3n/n)$. This result was subsequently refined by Frieze, Krivelevich, Michaeli, and Peled~\cite{Frieze-et-al}, who extended the bound to the asymptotically optimal $p = \omega(\log n/n)$. We consider the problem analogous to the $(n,d,\lambda)$-graphs. 
\begin{theorem}\label{thm:main-strong-cover}
    Let $\eps > 0$. There exists a constant $C = C(\eps)>0$ such that if $G$ is an $(n,d,\lambda)$-graph with $d/\lambda \geq C$, then $ST(G) \leq (1+\eps)n\log n$.
\end{theorem}
Again this is asymptotically optimal by the general lower bound of the cover time. In fact, we obtain~\Cref{thm:main-strong-cover} as a direct consequence of a stronger result. 

The connection between spectral expansion and the Hamiltonicity in pseudorandom graphs is well-established as in~\cite{Kriveham,glockham,ferber2024hamiltonicity,chen2025robustness}. A recent breakthrough by Draganić, Montgomery, Munhá Correia, Pokrovskiy and Sudakov~\cite{Draganic-et-al} demonstrates that any $(n,d,\lambda)$-graph admits a Hamiltonian cycle, if the spectral ratio $d/\lambda\ge C$ for a sufficiently large constant $C$. Drawing a parallel to the trace of random walk landscape, Frieze, Krivelevich, Michaeli, and Peled~\cite{Frieze-et-al} raised the question of whether the trace of a random walk on $(n,d,\lambda)$-graphs and random regular graphs exhibits similar Hamiltonicity properties. 
We answer these questions through the following two main results.

\begin{theorem}\label{thm:main}
    Let $\varepsilon > 0$. There exists a constant $C = C(\eps) >0$ such that the following holds: Let $G$ be an $(n,d,\lambda)$-graph with $d/\lambda \geq C$ and $L = (1+\eps) n\log n$. \textbf{whp} for every vertex $v\in V(G)$, the trace of a simple random walk starting at $v$ of length $L$ is Hamiltonian.
\end{theorem} 

We remark that implicitly in~\Cref{prop:cover-time},~\Cref{thm:main-strong-cover} and~\Cref{thm:main}, $d$ is sufficiently large. This follows from the fact that $\lambda = \Omega(d)$ (see~\cite{Krivelevich-Sudakov-survey}), which means $1/C \geq \lambda/d\geq \Omega(1/\sqrt{d})$ and therefore $d$ is sufficiently large as long as $C$ is.

It is well known that a random $d$-regular graph for any $d\geq 3$ is \textbf{whp} an $(n,d,\lambda)$-graph for $\lambda = O(\sqrt{d})$~\cite{Friedman}. We thus obtain the following corollary.
\begin{corollary}\label{cor:random-d-regular}
     Let $\varepsilon > 0$. There exists $d_0 = d_0(\eps)$ that the following holds: Let $G$ be a random $d$-regular graph with $d\geq d_0$ and let $L = (1+\eps) n\log n$. \textbf{whp} $G$ satisfies that \textbf{whp} for every vertex $v\in V(G)$, the trace of a simple random walk starting at $v$ of length $L$ is Hamiltonian.
\end{corollary}
We remark that both results are asymptotically optimal due to the general lower bound on the cover time. Our proof strategy is as follows.
First, since the second largest eigenvalue of the adjacency matrix of an $(n,d,\lambda)$-graph is bounded away from $d$, we can control the transitional probability going from any vertex $u$ to another vertex $v$ at step $t$ by using the spectral decomposition of the transition matrix.
We show that after $(1+\varepsilon)n\log n$ steps, each vertex is visited $\Theta(\log n)$ times with high probability. Then, we use these visit counts to establish the expansion properties of the trace graph $\Gamma$, which ensures the existence of a Hamiltonian cycle. 

Central to our proof of \Cref{thm:main} is the lemma that the trace of the random walk visits every vertex with logarithmic frequency. Specifically, \Cref{lem:number-of-vists} establishes that there exists a constant $\rho = \rho(\varepsilon) > 0$ such that with high probability, every vertex is visited at least $\rho \log n$ times. This phenomenon is intimately connected to the concept of \emph{blanket time} in the probability theory, a parameter introduced by Winkler and Zuckerman~\cite{Winkler-Zuckerman}. Let $\pi$ denote the stationary distribution, where $\pi_v = \deg(v)/(2|E(G)|)$. For a parameter $\delta \in (0,1)$, the $\delta$-blanket time, denoted $t_{bl}(G,\delta)$, is defined as the minimum time $t$ required for the walk to visit every vertex $v$ at least $\delta t \pi_v$ times. While it is immediate that $t_{bl}(G,\delta) \ge T(G)$, Winkler and Zuckerman conjectured that the blanket time is, in fact, within a constant factor of the cover time; that is, $t_{bl}(G,\delta) \le C_\delta T(G)$ for any $0<\delta < 1$ and any graph $G$. 
Kahn, Kim, Lov{\'a}sz and Vu~\cite{Kahn-Kim-Lovasz-Vu} showed that $t_{bl}(G,\delta)= O(T(G)(\log\log n)^2)$ before the conjecture was affirmatively settled in a celebrated work of Ding, Lee, and Peres~\cite{Ding-Lee-Peres}. 
Our results imply that there exists a $\delta = \delta(\eps)\in (0,1)$ for which $t_{bl}(G,\delta)\leq (1+\eps)n\log n$, which is the same as the cover time up to a $1+o(1)$ factor. In general, it is unclear for which graphs the same conclusion hold. In the concluding remarks we discuss and propose further questions along this direction.

\subsection{Paper organization} 
\Cref{sec:preliminaries} contains the background of random walks and tools we use.~\Cref{sec:proof} contains the proofs for all results and~\Cref{sec:concluding-remarks} some remarks and open problems.

\subsection*{Acknowledgement}
Part of this work was initiated during the visit of second author to Shanghai Center for Mathematical Sciences (SCMS) and Fudan University. The authors would like to thank Prof.\ Hehui Wu and his group memebers for their support and hospitality.

\section{Preliminaries}\label{sec:preliminaries}

\subsection{Notation}
For a graph $G$, we let $e(G)$ denote the number of its edges. For any two disjoint sets of vertices $S,T\subseteq V(G)$, we write $e_{G}(S,T)$ for the number of edges between $S$ and $T$ in $G$. For a set of vertices $S$, we denote by $S^c = V(G)\setminus S$.
We denote the maximum degree of a graph $G$ by $\Delta(G)$ and the neighborhood by $N(S) = \{v\in V(G)\setminus S: N(v)\cap S \neq \emptyset\}$. We write the degree of a vertex $v$ as $\deg(v)$.

For a simple random walk on graph $G$ starting from vertex $v$, let $\Gamma_v$ denote the \emph{trace graph}, which is a simple graph whose vertex set is $V(G)$ and edge set is the number of edges traversed by the random walk. When the starting vertex $v$ is irrelevant, we write $\Gamma$ instead.

All logarithms by default are base $e$. Floor and ceiling signs are omitted when there is no confusion.

\subsection{Basics of random walks}
Let $G$ be a connected simple graph. A simple random walk on $G$ starting at vertex $v_0\in V(G)$ is an infinite random sequence $(v_0,v_1,\dots)$, where at each step $i$, the walk move from $v_i$ to $v_{i+1}$, which is a uniformly random chosen neighbor of $v_i$. We denote such move by $v_i\rightarrow v_{i+1}$. The simple random walk has the Markov property, meaning the probability distribution at each step depends only on the current position and not on the history. 

Let $P = (P_{i,j})_{i,j=1}^n$ be the probability transition matrix, where $P_{i,j} = \Pr(v_i\rightarrow v_j) = 1/\deg(v_i)$ if $\{i,j\}\in E(G)$ and $0$ otherwise. Equivalently $P = D^{-1}A$, where $D$ is the diagonal matrix with $D_{i,i} = \deg(v_i)$ and $A$ is the adjacency matrix.
Note that the probability a random walk takes $t$ steps to reach vertex $v$ from vertex $u$ equals to $(P^t)_{u,v}$.
The \emph{stationary distribution} for a random walk $X$ on $G$ is a vector $\pi \in [0,1]^n$ such that $P\pi = \pi$. It is known that if $G$ is connected and non-bipartite, then $\pi$ exists and is unique (see e.g. ~\cite{Lovasz-survey}).
For $S\subseteq V(G)$, let $\pi(S) = \sum_{v\in S}\pi_v$. 
The \emph{total variation distance} between $X$ and $\pi$ is define as
\[
    \dtv(X_t, \pi) = \sup_{S\subseteq V(G)}|\Pr(X_t\in S) - \pi_S| = \frac{1}{2} \sum_{v\in V(G)}|\Pr(X_t = v) - \pi_v|\,.
\]
Let $Y$ be the stationary walk which is a simple random walk such that $\Pr(Y_0 = v) = \pi_v$ for every $v\in V(G)$.
There exists a standard coupling between $X$ and $Y$ such that for every $t$, we have
\begin{equation}\label{eqn:coupling}
     \Pr(\exists s\geq t: X_s\neq Y_s) \leq \dtv(X_t,\pi)\,.
\end{equation}

Now we define several well-known parameters of the random walk.
For a vertex $v\in V(G)$, let $C_v$ denote the expected number of steps for a simple random walk staring from $v$ to reach every vertex at least once. 
The \emph{cover time} is defined as $C_G = \max_{v\in V(G)}C_v$. 
The \emph{hitting time} $H(u,v)$ is defined as the expected number of steps for a random walk to reach $v$ from $u$. 
For $\xi\in (0,1)$, the \emph{$\xi$-mixing time} is defined as 
\[
    \tau(\xi) = \min\{t\geq 0: \dtv(X_s,\pi) <\xi \;\;\forall s\geq t\}\,.
\]
It is well-known that for expanders, the mixing time is small. For $(n,d,\lambda)$-graphs, a quantitative version is as follows (see for instance~\cite[Chapter 12]{Levin-Peres-book})
\begin{theorem}\label{thm:quick-mixing}
    Let $G$ be an $(n,d,\lambda)$-graph with $\lambda < d$. Then, 
    \[
         \tau(\xi) \leq \frac{\log n/2 + \log (1/2\xi)}{1-\lambda/d}\,.
    \]
    Assuming $d/\lambda \geq 100$, we have
    \[
     \tau(1/n) \leq 10 \log n\,.
    \]
\end{theorem}

\subsection{Electrical networks}
One may think of a (undirected, simple) graph as an electrical network where each edge carry $1$ Ohm of electrical resistance. The \emph{effective resistance} between two vertices $u,v$, denoted as $R_{u,v}$, is the amount of current flowing between $u$ and $v$ if we exert $+1$ volt of voltage at $u$ and $-1$ at $v$. 
Effective resistance is nonnegative.
One may compute the effective resistance by the series and parallel rules. Another, more systematic way, is to calculate effective resistance using the Laplacian matrix.

Recall the Laplacian matrix $L$ of a graph $G$ is defined to be $D-A$ where $D$ is the diagonal matrix with $D_{i,i} = \deg(v_i)$ and $A$ is the adjacency matrix and the pseudo-inverse of $L$, denoted as $L^\dagger$, is the inverse of $L$ restricted to the subspace orthogonal to the span of all-$1$ vector. Let $e_i$ be the $i$-th standard basis. Then, the effective resistance can be expressed as (see~\cite{Klein-Randic}):
\begin{equation}\label{eqn:Rij-Laplacian}
    R_{v_i, v_j} = (e_i - e_j)^T L^\dagger (e_i - e_j) = L^\dagger_{i,i} + L^\dagger_{j,j} - 2L^\dagger_{i,j}\,.
\end{equation}
A standard lower bound for $R_{u,v}$ is as follows (see e.g.~\cite{Jonasson}):
\begin{proposition}\label{prop:Ruv-lower-bound}
    For a graph $G$ and any two vertices $u,v\in V(G)$, we have 
    \[
        R_{u,v} \geq \frac{1}{d_u + 1} + \frac{1}{d_v+1}\,.
    \]
\end{proposition}

\subsection{Connection between parameters}
The effective resistance, hitting time and cover time are closely related.

Tetali~\cite{Tetali} showed the following connection between the effective resistance and the hitting time:
\begin{lemma}\label{lem:Tetali}
     For a graph $G$ and any two vertices $u,v\in V(G)$, we have 
     \[
        H_{u,v} = \frac{1}{2}\sum_{w\in V}\deg(w)\cdot (R_{u,v} - R_{u,w} + R_{v,w})\,.
     \]
\end{lemma}
Matthews~\cite{Matthews} proved the following connection between the hitting time and the cover time:
\begin{lemma}\label{lem:Matthews}
    For a graph $G$ and two vertices $u,v\in V(G)$, let $\mu_-$ and $\mu_+$ be real numbers such that $\mu_- \leq H_{u,v}\leq \mu_+$ for every pair of vertices $u,v$. Then,
    \[
         \mu_-\sum_{i=1}^n \frac{1}{i} \leq C_G\leq \mu_+ \sum_{i=1}^n \frac{1}{i}\,.
    \]
\end{lemma}

\subsection{Expander mixing lemma}
For $(n,d,\lambda)$-graphs, one of the most useful tools is the expander mixing lemma, proved by Alon and Chung~\cite{Alon-Chung-expander-mixing}.
\begin{lemma}[\cite{Alon-Chung-expander-mixing}]\label{lem:expander-mixing}
    Let $G$ be an $(n,d,\lambda)$-graph. Then, for any set $S\subseteq V(G)$ of size $s$, we have 
    \[
    |e_G(K) - ds^2/(2n)| \leq \lambda s/2\,.
    \]
    Moreover, for any two disjoint sets $S,T\subseteq V(G)$ of size $s,t$ respectively, we have
    \[
    |e_G(K,L) - dst/n| \leq \lambda\sqrt{st}\,.
    \]
\end{lemma}

\subsection{Hamiltonicity of expanders}
The main tool driving behind the proof is a recent breakthrough on the Hamiltonicity of pseudorandom graphs due to Draganić, Montgomery, Munhá Correia, Pokrovskiy and Sudakov~\cite{Draganic-et-al}. To state it, we first introduce the notion of $C$-expanders:
\begin{definition}
    A graph $G = (V, E)$ is a \emph{$C$-expander} if it satisfies two conditions: 
    \begin{itemize}
        \item (Expansion) Every subset $X$ with $1 \le |X| \le n/2C$ satisfies $|N_G(X)| \ge C|X|$
        \item (Joinedness) There exists at least one edge between any two disjoint sets of vertices of size $n/2C$.
    \end{itemize}
\end{definition}
In~\cite{Draganic-et-al}, they showed $C$-expanders are Hamiltonians for large enough $C$:
\begin{theorem}[Dragani\'{c} et al~\cite{Draganic-et-al}]\label{thm:Hamiltonicity}
    There exists a constant $C_0 >0$ such that if $G$ is a $C$-expander with $C\geq C_0$, then $G$ is Hamiltonian.
\end{theorem}

\subsection{Probabilistic inequalities}
Recall the Paley-Zygmund inequality:
\begin{proposition}\label{prop:Paley-Zygmund}
    Let $Z\geq 0$ be a random variable with finite variance. Then, 
    \[
        \Pr(Z > 0) \geq \frac{\E[Z]^2}{\E[Z^2]}\,.
    \]
\end{proposition}
We also need the following bound on the large deviation probability for a binomial random variable. It follows from the fact that the probability mass function of a binomially distributed random variable is increasing before the expectation:
\begin{proposition}\label{prop:bin-large-deviation}
    Let $X\sim \Bin(n,p)$. For $t\leq np$, we have
    \[
        \Pr(X\leq t) \leq t\binom{n}{t}p^t(1-p)^{n-t}\,.
    \]
\end{proposition}

\section{Proof}\label{sec:proof}
We first prove~\Cref{prop:cover-time}:
\begin{proof}[Proof of~\Cref{prop:cover-time}]
    From~\Cref{prop:Ruv-lower-bound}, we know $R_{u,v}\geq 2/(d+1)$. We now show that $R_{u,v}\leq 2/(d-\lambda)$.

    Let $d = \lambda_1\geq \cdots \geq \lambda_n$ be the eigenvalues of the adjacency matrix of $G$. Recall that $|\lambda_i|< \lambda$ for all $i\geq 2$.
    Consider the spectral decomposition of the Laplacian matrix: $L = \sum_{k=2}^n \mu_i u_i u_i^T$, where $\mu_i = d - \lambda_i \geq d-\lambda$ and $\{u_i\}_{i=2}^n$ forms an orthonormal basis. Thus, $L^\dagger = \sum_{k=2}^n (1/\mu_k)u_k u_k^T$.
    From~\eqref{eqn:Rij-Laplacian}, we know that 
    \[
        R_{i,j} = \sum_{k=2}^n \frac{1}{\mu_k}( (u_{k})_i - (u_k)_j)^2 \leq \frac{2}{d-\lambda}\,,
    \]
    where we used the orthonormality of $u_i$.
    Plugging this into~\Cref{lem:Tetali} and using the facts that $d$ and $d/\lambda$ are sufficiently large, we conclude
    \[
        (1-0.1\eps)n\leq \frac{1}{2} nd \left(\frac{4}{d+1} - \frac{2}{d-\lambda}\right)\leq H_{i,j} \leq \frac{1}{2}nd\left(\frac{4}{d-\lambda} -\frac{2}{d+1}\right) \leq (1+0.1\eps) n
    \]
    Then, applying~\Cref{lem:Matthews} with $\mu_- = (1-0.1\eps)n$ and $\mu_+ = (1+0.1\eps)n$ and using the fact that $\sum_{i=1}^n 1/i = (1+o(1))\log n$ finishes the proof.
\end{proof}

We now prove~\Cref{thm:main}.
The proof consists of two parts. First, we show that \textbf{whp} the trace visits every vertex at least $\rho \log n$ times for some $\rho = \rho(\eps)$. 
Using this, we show the trace graph $\Gamma$ is a $C'$-expander for some $C'$ that is sufficiently large as long as $C$ is, which then allows us to apply~\Cref{thm:Hamiltonicity}. 

\subsection{Number of visits}
For a random walk on a graph $G$ of length $L$ and a vertex $v\in V(G)$, let $\gamma(v)$ denote the number of times the random walk visits $v$.
The goal of this section is to prove the following:
\begin{lemma}\label{lem:number-of-vists}
    Let $\varepsilon > 0$. There exists constants $C = C(\eps) >0$ and $\rho = \rho(\eps) > 0$ such that the following holds: Let $G$ be a $(n,d,\lambda)$-graph with $d/\lambda \geq C$ and $L = (1+\eps) n\log n$. Consider a random walk on $G$ starting from an arbitrary vertex and of length $L$. Then, \textbf{whp} $\gamma(v)\geq \rho \log n$ for every vertex $v\in V(G)$.
\end{lemma}
We begin by establishing an auxiliary lemma, giving a lower bound on the probability a random walk starting from $u$ reaches $v$ in $n/\sqrt{C}$ steps:
\begin{lemma}\label{lem:lower-bound-return}
    Let $\varepsilon > 0$. There exists a constant $C = C(\eps) >0$ such that the following holds: Let $G$ be a $(n,d,\lambda)$-graph with $d/\lambda \geq C$ and $T = n/\sqrt{C}$. Then, for any pair of distinct vertices $u,v\in V(G)$, the probability a simple random walk $X$ starting from $u$ reaches $v$ at least once within $T$ steps is at least $(1-0.1\eps)/\sqrt{C}$.
\end{lemma}
\begin{proof}
    Let $N$ denote the number of visits to $v$ in the first $T$ steps and $I_t$ the indicator random variable for $X_t = v$ so that $N = \sum_{t = 1}^T I_i$. By linearity of expectation, we have 
    \[
        \E[N] = \sum_{t=1}^T \Pr(X_t = v|X_0 = u) = \sum_{t=1}^T (P^t)_{u,v}\,.
    \]
    Recall that the transition matrix $P = D^{-1}A$ and let $d = \lambda_1\geq \dots\geq \lambda_n$ are eigenvalues of $A$. Define $S = D^{1/2}PD^{-1/2} = D^{-1/2}AD^{-1/2}$. Since $S$ is a symmetric real square matrix, its spectral decomposition can be written as $S = \sum_{i=1}^n\mu_i\phi_i \phi_i^T$, where $\mu_i = \lambda_i/d\in [-1,1]$ are eigenvalues of $S$ and $\{\phi_i\}_{i=1}^n$ forms an orthonormal basis. 
    Observe that the stationary distribution of $X$ is $\pi = (1/n,\dots, 1/n)$, $\mu_1 = 1$ and $\phi_1(v) = (1/\sqrt{n},\dots, 1/\sqrt{n})$.

    Using the fact that 
    \[
        P^t = D^{-1/2}S^t D^{1/2} = S^t = \sum_{i=1}^n \mu_i^t \phi_i \phi_i^T\,,
    \]
    we conclude
    \[
        \E[N] = \frac{T}{n} + \sum_{t=1}^T \sum_{i=2}^n \mu_i^t\phi_i(u)\phi_i(v) = \frac{T}{n} + \sum_{i=2}^n \frac{\mu_i(1-\mu_i^T)}{1-\mu_i}\phi_i(u)\phi_i(v)\,.
    \]
    By Cauchy-Schwarz inequality and orthonormality of $\phi_i$, we know 
    \[
    \left|\sum_{i=2}^n \phi_i(u)\phi_i(v)\right|\leq \left(\sum_{i=2}^n \phi_i(u)^2\right)\left(\sum_{i=2}^n \phi_i(v)^2\right) \leq 1\,.
    \]
    Moreover, using the fact that $|\mu_i| = |\lambda_i|/d \leq \lambda/d \leq 1/C$ for all $i\geq 2$, we conclude 
    \[
        \left|\frac{\mu_i(1-\mu_i^T)}{1-\mu_i}\right|\leq   \left|\frac{\mu_i}{1-\mu_i}\right|\leq \frac{1/C}{1-1/C} \leq 2/C\,.
    \]
    Thus, $\E[N]\geq T/n - 2/C$.

    Now we compute $\E[N^2]$. Observe that
    \[
        \E[N^2] = \E\left[\left(\sum_{i=1}^T I_t\right)^2\right] = \E[N] + 2\sum_{1\leq t_1<t_2\leq T}\E[I_{t_1}I_{t_2}]\,.
    \]
    Using the Markov's property, we have
    \[
        \E[I_{t_1}I_{t_2}] = \Pr(\{X_{t_1} = v\}\cap \{X_{t_2} = v\}\mid X_0 = u) = \Pr(X_{t_1} = v\mid X_0 = u)\cdot \Pr(X_k = v\mid X_0 = v)\,,
    \]
    where $k = t_2-t_1\geq 1$.
    
    Let $R_v(T) = \sum_{t=1}^T(P^t)_{v,v}$. Since $R_v(T)$ is monotone non-decreasing in $T$, we conclude 
    \begin{align*}
        \E[N^2] &= \E[N] + 2\sum_{t_1=1}^T \Pr(X_{t_1} = v\mid X_0= u)\sum_{k=1}^{T-t_1}\Pr(X_k = v\mid X_0 =v)\\
        &= \E[N] +  2\sum_{t_1=1}^T\Pr(X_{t_1}=v\mid X_0 = u)\cdot  R_{v}(T-t_1)\\
        &\leq \E[N] + 2\E[N] \cdot R_v(T) = (1+ 2R_v(T))\cdot \E[N]\,.
    \end{align*}
    By a similar computation as for $\E[N]$, we know that
    \[
        R_v(T) = \frac{T}{n} + \sum_{i=2}^n \frac{\mu_i(1-\mu_i^T)}{1-\mu_i}\phi_i(v)^2 \leq \frac{T}{n} + \frac{2}{C} \,.
    \]
    Applying~\Cref{prop:Paley-Zygmund} and using $T = n/\sqrt{C}$ for a sufficiently large $C$, we conclude 
    \[
        \Pr(N > 0) \geq \frac{\E[N]^2}{\E[N^2]} \geq \frac{\E[N]}{2R_v(T) + 1} \geq \frac{T/n-2/C}{2(T/n + 2/C)+1}\geq \frac{1}{\sqrt{C}}(1-0.1\eps)\,,
    \]
    as desired.
\end{proof}

We have the following immediate corollary for the stationary walk:
\begin{corollary}\label{cor:stationary-walk}
    Let $\varepsilon > 0$. There exists a constant $C = C(\eps) >0$ such that the following holds: Let $G$ be a $(n,d,\lambda)$-graph with $d/\lambda \geq C$ and $T = n/\sqrt{C}$. Then, for any vertex $v\in V(G)$, the probability a stationary walk visits $v$ at least once within $T$ steps is at least $(1-0.1\eps)/\sqrt{C}$.
\end{corollary}
\begin{proof}
    Recall that the stationary distribution for a $d$-regular graph is the uniform distribution over all vertices $(1/n,\dots, 1/n)$. Also, the stationary walk $Y$ and the simple random walk $X$ have the same distribution conditionally on having the same starting vertex. 
    Since~\Cref{lem:lower-bound-return} holds for every starting vertex $u\in V(G)$, the conclusion follows by integrating over the starting vertex. 
\end{proof}

We are now ready to prove~\Cref{lem:number-of-vists}:
\begin{proof}[Proof of~\Cref{lem:number-of-vists}]
    Let $X$ denote the simple random walk and $Y$ denote the stationary walk.
    Recall $L = (1+\eps) n\log n$ and $T = n/\sqrt{C}$. Let $b = \tau(1/n)$. By~\Cref{thm:quick-mixing}, we know $b\leq 10 \log n$.
    We divide $X$ into consecutive segments of length $10\log n + T$ and treat each segment separately. For simplicity, the subscript here are with respect to the starting point in each segment.
    Within each segment, by the definition of the mixing time, 
    we know $\dtv(X_t, Y_t)\leq 1/n$ for any $t\geq 10\log n$. Using~\eqref{eqn:coupling}, we conclude $$\Pr(\exists 10\log n\leq s\leq 10\log n+T: X_s\neq Y_s)\leq 1/n\,.$$
    Fix an arbitrary vertex $v\in V(G)$.
    By Markov's property and~\Cref{cor:stationary-walk}, for each segment, the walk visits $v$ with probability at least $(1-0.1\eps)/\sqrt{C} - 1/n \geq (1-0.2\eps)/\sqrt{C}$. Thus, by~\Cref{prop:bin-large-deviation}, the probability $v$ is visited fewer than $\rho \log n$ times within the $L$ steps is at most 
    \begin{align*}
        &\Pr\left(\Bin\left(\frac{L}{10\log n + T}, \frac{1-0.2\eps}{\sqrt{C}}\right)\leq 
        \rho \log n \right)\\
        &\quad\leq \rho \log n \cdot \binom{L/(10\log n+T)}{\rho \log n} ((1-0.2\eps)/\sqrt{C})^{\rho \log n} ( 1-  (1-0.2\eps)/\sqrt{C})^{L/(10\log n +T) - \rho\log n} \\
        &\quad=o(n^{-1-0.1\eps})\,,
    \end{align*}
    if we choose $\rho$ sufficiently small and $C$ sufficiently large.
    The conclusion follows from a union bound over all vertices $v\in V(G)$.
\end{proof}

\subsection{C-expander}
The goal of this section is to prove the following:
\begin{lemma}\label{lem:C-expander}
     Let $\varepsilon > 0$. There exists a constant $C = C(\eps) >0$ such that the following holds: Let $G$ be a $(n,d,\lambda)$-graph with $d/\lambda \geq C$ and $L = (1+\eps) n\log n$. 
     Let $\Gamma$ be the trace graph obtained by a random walk of length $L$ starting from any vertex. Then,
     \textbf{whp} $\Gamma$ is a $C'$-expander for $C' = \log \log C$.
\end{lemma}
\begin{proof}
    To prove a graph is $C$-expander, we need to verify the expansion and joinedness properties. By~\Cref{lem:number-of-vists}, there exists a constant $\rho > 0$ such that $\gamma(v) \geq \rho \log n$ for all $v\in V(G)$. 
    
    We start by proving the expansion property. Suppose for contradiction there exists a set $A\subseteq V(G)$ of size $a\leq n/(2C')$ such that $|N_\Gamma(A)|\leq C'a$. For $v\in A$, let $d_v \coloneqq |N_\Gamma(v)|$. Note that the existence of such a set $A$ implies for every vertex $v\in A$, every time the random walk exits $v$, it goes to one of the $d_v$ neighbors of $v$ in $N_\Gamma(v)$, which happens with probability $d_v/d$ independently. Using the independence between exits and that $\gamma(v)\geq \rho \log n$, we conclude the probability for a fixed outcome of $A$ and $N_\Gamma(A)$ satisfying the conditions above to happen is at most
    $$\prod_{v\in A}\left(\frac{d_v}{d}\right)^{\rho \log n} = \biggl(\prod_{v\in A}\frac{d_v}{d}\biggr)^{\rho \log n}\,.$$
    By~\Cref{lem:expander-mixing}, $|N_\Gamma(A)|\leq C'a$ and $\lambda \leq d/C$, we know 
    \[
        \sum_{v\in A}d_v \leq e_G(A, N_\Gamma(A)) \leq \frac{d}{n} C'a^2 + \frac{d\sqrt{C'}}{C}a\,.
    \]
    When $a\geq 10n/(C\sqrt{C'})$, we have $\sum_{v\in A}d_v \leq 1.1dC'a^2/n$ and otherwise we have $\sum_{v\in A}d_v \leq 11d\sqrt{C'}a/C$. Applying the inequality of arithmetic and geometric means, we know $\prod_{v\in A}d_v\leq (\sum_{v\in A}d_v/a)^{a}$.
    Thus, 
    \[
        \prod_{v\in A}d_v\leq \begin{cases}
            (1.1dC'a/n)^a&  \text{if } a\geq 10n/(C\sqrt{C'})\\
           (11d\sqrt{C'}/C)^{a} & \text{otherwise}  
            
        \end{cases}
    \]
    When $a\geq 10n/(C\sqrt{C'})$, the number of such sets $A$ and $N_\Gamma(A)$ is upper bounded by $2^{2n}$. Using the fact that $a\leq n/(2C')$, the probability there exists such set $A$ is at most 
    \[
        2^{2n} (1.1C'a/n)^{a\rho \log n}\leq 2^{2n} \exp\left(\frac{10\rho}{C\sqrt{C'}}\log_2 (1.1/2) \cdot n\log n\right) = o(1)\,.
    \]
    Otherwise, the number of such sets $A$ and $N_\Gamma(A)$ is upper bounded by $\binom{n}{a}\binom{da}{C'a}$. Thus, the probability there exists such set $A$ is at most 
    \[
        \binom{n}{a}\binom{da}{C'a} (11\sqrt{C'}/C)^{a\rho \log n} \leq \exp\bigl(a\log n + C'a\log (d/C') + \log(11\sqrt{C'}/C)\rho a\log n\bigr) = o(1)\,,
    \] since $C' = \log\log C$ and $C$ is sufficiently large. In both cases, the expansion property is verified. 

    Now we verify the joinedness property. Suppose for contradiction there exists two disjoint sets $A$ of size $a$ and $B$ of size $b$ where $a = b = n/2C'$ and that $e_\Gamma(A,B) = 0$.
    By~\Cref{lem:expander-mixing} and $\lambda\leq d/C$ and $C'=\log\log C$, we know 
    $$e_G(A,B) \geq \frac{d}{n}ab - \lambda\sqrt{ab} =  \frac{dn}{4(C')^2} - \frac{\lambda n}{2C'} \geq \frac{dn}{10(C')^2}\,.$$
    Let $A' = \{v\in A: N_G(v)\cap B \geq d/(10C')\}$. Then,
    \[
        \frac{dn}{10(C')^2}\leq e_G(A,B) \leq d |A'| + \frac{d}{10C'}\left(\frac{n}{2C'} - |A'|\right) \implies |A'|\geq \frac{n}{30(C')^2}\,.
    \]
    Recall that each vertex is visited at least $\rho \log n$ times. Thus, the probability $e_\Gamma(A,B) = 0$ is at most the probability none of the edges between $A'$ and $B$ survives, which is at most 
    \[
        (1-1/(10C'))^{\rho n \log n/(30(C')^2)} \leq \exp\left(-\frac{\rho}{9000(C')^3} n\log n\right) = o(2^{-2n}).
    \]
    It remains to union bound over all such $A$ and $B$ for which there are at most $2^{2n}$ possibilities.
\end{proof}

\subsection{\texorpdfstring{Proof of~\Cref{thm:main}}{Proof of Theorem~\ref{thm:main}}}
We are now ready to prove~\Cref{thm:main}:
\begin{proof}[Proof of~\Cref{thm:main}]
    Let $C_0$ be as in~\Cref{thm:Hamiltonicity}.
    From~\Cref{lem:C-expander}, we know \textbf{whp} $\Gamma$ is a $C'$-expander. By taking $C$ to be sufficiently large, we may assume $C'= \log\log C \geq C_0$. Thus, we may apply~\Cref{thm:Hamiltonicity} to deduce that $\Gamma$ is Hamiltonian, as desired. 
\end{proof}

\section{Concluding remarks}\label{sec:concluding-remarks}
One of the key tools we applied was~\Cref{thm:Hamiltonicity}, which requires only $C$-expanders. One may wonder if~\Cref{thm:main} is true with $(n,d,\lambda)$-graphs replaced by $C$-expanders. This is unfortunately (very) incorrect:
\begin{proposition}
    For any $C \leq 1.1n/\log n$, there exists a $C$-expander for which the cover time is $a n \log n$ for some $a > 1$.
\end{proposition}
\begin{proof}[Proof sketch]
    Consider a complete graph on $n-2$ vertices together with two special vertices $u,v$, each connecting to $C$ vertices of the complete graph and their neighborhoods are disjoint. Clearly, the graph is a $C$-expander. However, note that the expected number of steps needed to go from $u$ to $v$ is at least $1 + (n-1)/C)\cdot (n-1)> n\log n$. Indeed, the $1$ comes from going from $u$ to the clique. Then, in each step, it will have at most a $C/(n-1)$ chance of landing in $N(v)$ and thus in expectation $(n-1)/C$ steps are needed to land in $N(v)$. Starting from a vertex in $N(v)$, each time there is a $1/(n-1)$ chance of going to $v$ so in expectation it takes $n-1$ such trials. Clearly, the expected number of steps taken to go from $u$ to $v$ is a lower bound on the cover time, which finishes the proof.
\end{proof}
One natural strengthening of~\Cref{thm:main} would be to consider following hitting time problem: Let $\tau_{HC}(G)$ be the minimum $t\in \mathbb N$ for which the trace graph $\Gamma$ becomes Hamiltonian and $\tau_{1}$ the minimum $t\in \mathbb N$ such that $\Gamma$ has minimum degree $1$.
For $(n,d,\lambda)$-graphs with $d/\lambda \geq C$ or a random $d$-regular graph for $d$ sufficiently large, is it true that \textbf{whp} $\tau_{HC} = \tau_1 + 1$? Note that $\tau_{HC}\geq \tau_1 + 1$ is necessary since a Hamilton cycle uses two incident edges from each vertex.
For complete graph the corresponding hitting time result is known~\cite{Frieze-et-al} .

Another natural direction is to consider relations between cover time, strong cover time and the blanket time. For $(n,d,\lambda)$-graph with $d/\lambda$ sufficiently large, in~\Cref{thm:main-strong-cover}, we have established that the cover time and the strong cover time differ only by a factor of $1+o(1)$. In general, can the strong cover time always be controlled by the cover time by a constant factor? 
As mentioned in the introduction, ~\Cref{lem:number-of-vists} implies the existence of a $\delta = \delta(\eps)\in (0,1)$ such that the $\delta$-blanket time $t_{bl}(G,\delta)$ and the cover time differ only a factor of $1+o(1)$ for $(n,d,\lambda)$-graphs with $d/\lambda$ sufficiently large. Are there other natural graph classes that the same conclusion holds? When the cover time of a graph is $(1+o(1))n\log n$ (almost match the general lower bound $(1-o(1))n\log n)$, does such $\delta$ necessarily exist?

\end{document}